\documentclass[11pt]{amsart}
\usepackage{amsmath}
\usepackage{amssymb, verbatim}
\usepackage{pstricks,pst-node,pst-plot}
\usepackage{comment}
\usepackage{color}
\usepackage{extarrows}
\usepackage{graphicx}
\usepackage{tikz}
\usetikzlibrary{arrows,positioning,shapes}

\title[ The deformed Hermitian-Yang-Mills equation
on the blowup of $\mathbb P^n$]{The deformed Hermitian-Yang-Mills equation
on the blowup of $\mathbb P^n$}

\author[A. Jacob]{Adam Jacob*}
\thanks{$^{*}$Supported in part by a Simons Collaboration Grant.}
 \address{Department of Mathematics, University of California Davis, 1 Shields Ave., Davis, CA, 95616}
 \email{ajacob@math.ucdavis.edu}
\author[N. Sheu]{Norman Sheu}
 \email{normansheu@math.ucdavis.edu}

\numberwithin{equation}{section}

\renewcommand{\thanks}[1]{\footnote{#1}} 

\newcommand{\be}{\begin{equation}}
\newcommand{\bea}{\begin{eqnarray}}
\newcommand{\eea}{\end{eqnarray}} \newcommand{\ee}{\end{equation}}
 
 \def\ba{\begin{eqnarray}}
\def\ea{\end{eqnarray}}

\def\ra{\rightarrow}

\def\al{\alpha}

\def\ti{\tilde}

\def\ra{\rightarrow}

\def\[{{\bf [}}
\def\]{{\bf ]}}

\begin{document}
\newtheorem{theorem}{Theorem}
\newtheorem{proposition}{Proposition}
\newtheorem{lemma}{Lemma}
\newtheorem{corollary}{Corollary}
\newtheorem{definition}{Definition}
\newtheorem{conjecture}{Conjecture}
\newtheorem{example}{Example}
\newtheorem{claim}{Claim}
\newtheorem{assumption}{Assumption}
\theoremstyle{plain}

\maketitle

\vspace{.1in}

\begin{abstract}
We study the deformed Hermitian-Yang-Mills equation on the blowup of complex projective space.  Using symmetry, we express the equation as an ODE which can be solved using combinatorial methods if an algebraic    stability condition is satisfied.  This gives  evidence towards a conjecture of the first author, T.C. Collins, and S.-T. Yau on general compact K\"ahler manifolds.  \end{abstract}

\begin{normalsize}

\section{Introduction}

This paper explores the relationship between stability and solutions to the deformed Hermitian-Yang-Mills  equation.  Let $(X,\omega)$ be a compact K\"ahler manifold, and $[\alpha]\in H^{1,1}(X,\mathbb R)$ a real cohomology class.  The class $[\alpha]$ solves the   deformed Hermitian-Yang-Mills  equation if it admits  a representative $\alpha\in[\alpha]$ satisfying
\be
\label{dHYM1}
{\rm Im}(e^{- i\hat\theta}(\omega+i\alpha)^n)=0,
\ee 
where $e^{i\hat\theta}\in S^1$ is a fixed constant. Fixing  $\alpha_0\in [\alpha]$, by the $\partial\bar\partial$-Lemma, any other representative of this class can be written as $\alpha=\alpha_0+i\partial\bar\partial \phi$ for some real function $\phi$, and so \eqref{dHYM1} is an elliptic, fully nonlinear equation for $\phi$.

A complex analogue of the special Lagrangian graph equation,  equation \eqref{dHYM1} was derived by Mari\~no-Minasian-Moore-Strominger by studying equations of motion for BPS $B$-branes \cite{MMMS}.  Taking a more geometric viewpoint, Leung-Yau-Zaslow   derived this equation by looking at the mirror of special Lagrangian graphs under the semi-flat setup of SYZ  mirror symmetry \cite{LYZ}. Recently, the question of how existence of solutions to  dHYM equation may relate to various algebraic stability conditions has garnered  significant attention, due to exciting relationships with other equations arising in complex geometry, and furthermore due to how such stability conditions may shed light on the existence problem for special Lagrangian submanifolds in Calabi-Yau manifolds.

Initial attempts to solve equation \eqref{dHYM1} were undertaken in \cite{JY} and later \cite{CJY}, and relied on certain analytic assumptions, namely that the class $[\alpha]$ admitted a representative that satisfied a positivity condition. This lead  to the natural question of whether solvability can be determined by an algebraic  condition on the classes $[\alpha]$ and $[\omega]$ alone. Following the work of  Lejmi-Sz\'ekelyhidi on the $J$-equation \cite{LS}, the first author, along with T.C. Collins and   S.-T. Yau, integrated the positivity condition along subvarieties to develop a necessary class condition for existence, and conjectured it was a sufficient condition as well  \cite{CJY}. We formally state this conjecture. First, for an analytic subvariety $V\subseteq X$,   define the complex number
\be
Z_{[\alpha][\omega]}(V):=-\int_V e^{-i\omega+\alpha},\nonumber
\ee
where by convention we only integrate the term in the expansion of order ${\rm dim}(V)$. By the $\partial\bar\partial$-Lemma $Z_{[\alpha][\omega]}(V)$ is independent of a choice of representative from $[\omega]$ or $[\alpha]$.  The main results of \cite{CJY} rely on an assumption referred to as {\it supercritical phase}, which assumes that the constant $\hat\theta$ can be lifted to $\mathbb R$ to lie within the interval $((n-2)\frac\pi 2,n\frac\pi2)$. Therefore we state the conjecture with this assumption:

\begin{conjecture}[Collins-J-Yau \cite{CJY}] \label{theconjecture}
The  cohomology class $[\alpha]\in H^{1,1}(X,\mathbb R)$ on a compact K\"ahler manifold $(X,\omega)$ admits a solution to the deformed Hermitian-Yang-Mills equation \eqref{dHYM1} (with supercritical phase) if and only if $Z(X)\neq 0$, and for all analytic subvarieties $V\subset X$,
\be
\label{stabconj}
{\rm Im}\left(\frac{Z_{[\alpha][\omega]}(V)}{Z_{[\alpha][\omega]}(X)}\right)>0.
\ee
\end{conjecture}
Without the supercritical phase assumption a stability conjecture can still be formulated, although as opposed to \eqref{stabconj} the inequality will be of a slightly  different form, as discussed below.

Following the above work, Collins-Yau subsequently  constructed a more robust necessary condition for existence, for which the above conjecture is only a special case \cite{CY}. Their approach follows an infinite dimensional GIT picture, and looks at the limiting behavior of geodesics in the space of potentials for $[\alpha]$, in conjunction with the behavior of various functionals.  Overall, the viewpoint of this work is   that any stability condition for \eqref{dHYM1}  should arise naturally as an obstruction to existence.  Colins-Yau also relate their work to other conjectured  stability conditions for similar problems, including Bridgeland stability. We direct the interested reader to \cite{CY} for more details on their stability condition as it relates to Bridgeland stability,  and instead only focus on Conjecture \ref{theconjecture}.  

In this paper work on the blowup of complex projective space. We find a stability condition, which is a generalization of \eqref{stabconj} in the non supercritical phase case, and demonstrate that stability  is sufficient for existence of a solution.
\begin{theorem}
\label{maintheorem}
Let $X$ be the blowup of $\mathbb P^n$ at a point. Let $[\omega]$ be any K\"ahler class on $X$, and  $[\alpha]$ any real cohomology class. Then if  $Z(X)\neq 0$, and if for each $k\in\{1,...,n-1\}$   all analytic subvarieties $V^k\subset X$ of dimension $k$ satisfy either
\be
\label{stabconj2}
{\rm Im}\left(\frac{Z_{[\alpha][\omega]}(V^k)}{Z_{[\alpha][\omega]}(X)}\right)>0\qquad{\rm or}\qquad {\rm Im}\left(\frac{Z_{[\alpha][\omega]}(V^k)}{Z_{[\alpha][\omega]}(X)}\right)<0,
\ee
then $[\alpha]$ admits a solution to the deformed Hermitian-Yang-Mills equation. 
\end{theorem}
We reiterate that for different dimensions $k$, we allow for the inequality in \eqref{stabconj2} to be either positive or negative. However, for a fixed $k$, all subvarieties of that dimension  must give the same sign. We note that in the supercritical phase case, only the strictly positive inequality is possible,  and so our condition \eqref{stabconj2} reduces to  \eqref{stabconj}, proving  Conjecture \ref{theconjecture} in this case.

To prove our theorem, we make use of the fact that on $X$, both $[\omega]$ and $[\alpha]$ admit representatives that satisfy a particular symmetry called {\it Calabi Symmetry}. Originally studied by Calabi to construct examples of extremal K\"ahler metrics \cite{Ca}, this symmetry has since been employed to study many other geometric equations, including the K\"ahler Ricci flow \cite{Song,SW2,SW3,SW4}, metric flips \cite{SY}, and the inverse $\sigma_k$ equations \cite{FL}. The advantage of working with Calabi Symmetry is that allows us to write equation \eqref{dHYM1} as an ODE over a closed interval in $\mathbb R$,  with a two sided boundary conditions determined by the classes $[\omega]$ and $[\alpha]$. Thus the question of existence  is reduced to solving the boundary valued ODE. Of course, by existence and uniqueness of solutions to ODEs we can always find a solution matching one boundary value, so the difficulty is determining when the other boundary value matches up. This is where stability comes into play, and we use \eqref{stabconj2} to force the boundary values into certain configurations where a solution will always exist.

While this theorem demonstrates that \eqref{stabconj2} is a sufficient condition for existence, it is not clear it is necessary. As noted above, outside of the supercritcal phase case,   \eqref{stabconj2} does not match the necessary condition  for existence presented in \cite{CJY}. To elaborate, let the average angle of a subvariety $V^k$ be defined by the argument of  $\int_{V^k}(\omega+i\alpha)^k$, and  denote this argument by $\hat\Theta_{V^k}$. In   \cite{CJY} it is demonstrated that  any class that solves \eqref{dHYM1} must satisfy
\be
 \hat\Theta_{V^{k}} >\hat\theta-(n-k)\frac\pi2.\nonumber
\ee
In fact, assuming supercritical phase the above inequality is equivalent to \eqref{stabconj}. However, outside of   supercritical phase, one needs to specify a unique lift of $\hat\theta$ to $\mathbb R$, before a necessary condition similar to the above can be generalized. If such a lift exists, then again a solution to equation \eqref{dHYM1} will imply
\be
\label{liftedstab}
\hat\theta+(n-k)\frac\pi2>\hat\Theta_{V^{k}} >\hat\theta-(n-k)\frac\pi2.
\ee
When $k=n-1$, we find  the above inequality is a stronger condition than \eqref{stabconj2}, whereas for $k<n-1$ the conditions fail to match. Nevertheless, we are able to demonstrate:
\begin{theorem}
\label{theorem2}
Let $X$ be the blowup of $\mathbb P^n$ at a point. Let $[\omega]$ be any K\"ahler class on $X$, and  $[\alpha]$ any real cohomology class. Then  $[\alpha]$ admits a solution to the deformed Hermitian-Yang-Mills equation if and only if 
\begin{enumerate}
    		\item The average angle $\hat\theta$ has a unique lift to $\mathbb R.$
    		\item For every divisor $V^{n-1}\subset X$, the average angle $\hat\Theta_{V^{n-1}}$ satisfies \eqref{liftedstab}.   	
		\end{enumerate}
\end{theorem}
Here we see  the importance of  finding a lift of $\hat\theta$. In general, finding a purely algebraic method for lifting $\hat\theta$, which only depends on the classes $[\omega]$ and $[\alpha]$,  would greatly aid our understanding of the relationship between  solvability of \eqref{dHYM1} and stability. In this light, one could view condition \eqref{stabconj2} as algebraic condition which specifies a lift of $\hat\theta$, which then leads to a solution of the equation.  Therefore, it would be interesting to develop more such methods of lifting $\hat\theta$ in general.

The paper is organized as follows. In Section \ref{2} we reformulate equation \eqref{dHYM1}  and introduce  the Calabi Symmetry ansatz,  and show how solutions to \eqref{dHYM1} correspond to solutions of an exact ODE. In Section \ref{3} we explicitly compute the inequalities arising from the stability condition \eqref{stabconj2} for all subvarieties   of $X$. We then show how these  inequalities define regions in $\mathbb R^2$ where the graph of our ODE is given, and prove a key proposition relating the slopes of the boundaries of these regions. This proposition is used in Section \ref{4} to limit the initial configurations of boundary values for our ODE, which we use  to prove Theorem \ref{maintheorem}. We conclude the paper in Section \ref{5} with a discussion on how   $\hat\theta$ can be lifted from $S^1$ to $\mathbb R$ without appealing to existence of a solution, assuming \eqref{stabconj2} is satisfied for all subvarieties. We then prove Theorem \ref{theorem2}.
\medskip

{\bf Acknowledgements.} We would like to thank Tristan C. Collins for many valuable  discussions and comments. This work was funded in part by a Simons collaboration grant. 

\section{Background and Calabi Symmetry}
\label{2}

Let $(X,\omega)$ be a compact K\"ahler manifold, and $[\alpha]\in H^{1,1}(X,\mathbb R)$ a real cohomology class.  We study the {\it deformed Hermitian-Yang-Mills} equation, which as stated in the introduction seeks a representative $\alpha\in[\alpha]$ satisfying
\be
{\rm Im}(e^{- i\hat\theta}(\omega+i\alpha)^n)=0\nonumber
\ee 
for a fixed constant $e^{i\hat\theta}\in S^1$. Integrating the above equation we see the angle $\hat\theta$ must be the argument of the complex number
\be
\zeta_X:=\int_X(\omega+i\alpha)^n.\nonumber
\ee
By the $\partial\bar\partial$-Lemma $\zeta_X$ is independent of a choice of representatives of the classes $[\omega]$ and $[\alpha]$. Thus we see a simple necessary class condition for existence is that  $\zeta_X\neq 0$.

We reformulate the deformed Hermitian-Yang-Mills equation  as follows. Given a representative $\alpha\in[\alpha]$, let $\lambda_1,...,\lambda_n$ denote the real eigenvalues of the Hermitian endomorphism $\omega^{-1}\alpha$. Then, at a fixed point where $\omega^{-1}\alpha$ is diagonal, we see
\be
{\rm Im}\left(e^{- i\hat\theta}\frac{(\omega+i\alpha)^n}{\omega^n}\right)={\rm Im}\left(e^{- i\hat\theta}\prod_{k=1}^n(1+i\lambda_k)\right).\nonumber
\ee 
We denote the angle of the complex number $\prod_{k=1}^n(1+i\lambda_k)$ by $\Theta_\omega(\alpha)$, which can be computed as follows:
\bea
\Theta_\omega(\alpha)&=&-i{\rm log}\frac{\prod_{k=1}^n(1+i\lambda_k)}{|\prod_{k=1}^n(1+i\lambda_k)|}\nonumber\\
&=&-i{\rm log}\frac{\prod_{k=1}^n(1+i\lambda_k)}{(\prod_{k=1}^n(1+i\lambda_k)\prod_{k=1}^n(1-i\lambda_k))^{\frac12}}\nonumber\\
&=&-\frac i2{\rm log}\frac{\prod_{k=1}^n(1+i\lambda_k)}{\prod_{k=1}^n(1-i\lambda_k)}.\nonumber
\eea
By the complex formulation of arctangent, we arrive at
\be
\Theta_\omega(\alpha)=\sum_{k=1}^n{\rm arctan}(\lambda_k).\nonumber
\ee
Thus equation \eqref{dHYM1} is equivalent to
\be
\label{dHYM2}
\Theta_\omega(\alpha)=\hat\theta\qquad{\rm mod}\,\,2\pi.
\ee
The advantage of this formulation is that the pointwise angle $\Theta_\omega(\alpha)$ is real valued  and lies in $(-n\frac \pi2, n\frac \pi 2)$, while $e^{i\hat\theta}$ is only valued in $S^1$. Thus a solution of the deformed Hermitian-Yang-Mills equation specifies a unique lift of $\hat\theta$ to $\mathbb R$.  We refer to such a lift  as a {\it branch} of the equation.

In this paper we construct solutions to the deformed Hermitian-Yang-Mills equation in a specific geometric setup, where we can take advantage of large symmetry. Specifically, let $X$ be the K\"ahler manifold defined by blowing up $\mathbb P^n$ at one point $x_0$. Let $E$ denote the  exceptional divisor, and $H$  the pullback of the hyperplane divisor from  $\mathbb P^n$. These two divisors span $H^{1,1}(X,\mathbb R)$, and any K\"ahler class will lie in $a_1[H]-a_2[E]$ with $a_1>a_2>0$. Normalizing,  assume $X$ admits a K\"ahler form $\omega$ in the class 
\be
[\omega]= a[H]-[E],\nonumber
\ee
with $a>1$. Furthermore, assume our class $[\alpha]$ satisfies
\be
[\alpha]= p[H]-q[E],\nonumber
\ee
for a choice of $p,q\in\mathbb R$. 

Calabi introduced the following ansatz in \cite{Ca}. On $X\backslash (H\cup E)\cong \mathbb C^n\backslash \{0\}$ define the radial coordinate
\be
\rho={\rm log}(|z|^2).\nonumber
\ee
Any function $u(\rho)\in C^\infty(\mathbb R)$ that satisfies $u'(\rho)>0,$ $u''(\rho)>0$, has the property that its complex Hessian $\omega=i\partial\bar\partial u$  defines a K\"ahler form on $\mathbb C^n\backslash \{0\}$. In order for $\omega$ to extend to a K\"ahler form on $X$ in the class $a[H]-[E]$, we need $u$ to satisfy the following boundary asymptotics. Define the functions $U_0, U_\infty:[0,\infty)\rightarrow \mathbb R$ via 
\be
U_0(r):= u({\rm log}r)-{\rm log}r\qquad {\rm and}\qquad U_\infty(r):= u(-{\rm log}r)+a{\rm log}r.\nonumber
\ee
Then we need both $U_0$ and $U_\infty$ to extend by continuity to a smooth function at $r=0$, with both $U_0'(0)>0$ and  $U_\infty'(0)>0$. In particular this fixes the following asymptotic behavior of $u$:
\be
\lim_{\rho\ra-\infty}u'(\rho)=1,\qquad\lim_{\rho\ra\infty}u'(\rho)=a.\nonumber
\ee
This ensures that $\omega=i\partial\bar\partial u$ extends to a K\"ahler form on $X$ and lies in the correct class.

 Similarly, for any function $v(\rho)\in C^\infty(\mathbb R)$, the Hessian $i\partial\bar\partial v(\rho)$ defines a $(1,1)$ form $\alpha$ on $\mathbb C^n\backslash \{0\}$. In order for $\alpha$ to extend to $X$ in the class $[\alpha]$,  we require asymptotics of the same form, without any positivity assumptions since $[\alpha]$ need not be a K\"ahler class. As above, we define the functions    $V_0, V_\infty:[0,\infty)\rightarrow \mathbb R$ via 
\be
V_0(r):= v({\rm log}r)-q{\rm log}r\qquad {\rm and}\qquad V_\infty(r):= v(-{\rm log}r)+p{\rm log}r,\nonumber
\ee
and  specify that $V_0$ and $V_\infty$  extend by continuity to a smooth function at $r=0$. As a result $v(\rho)$ satisfies:   
\be
\label{vinfinity}
\lim_{\rho\ra-\infty}v'(\rho)=q,\qquad\lim_{\rho\ra\infty}v'(\rho)=p.
\ee
Then $i\partial\bar\partial v $ extends to a smooth (1,1) form on $X$ in the class $[\alpha]$.

Given this setup, the deformed Hermitian-Yang-Mills equation reduces to an ODE. In particular, for a given function $u(\rho)$ satisfying the Calabi ansatz above (which defines our background K\"ahler form), we need to find a function $v(\rho)$ of a single real variable $\rho$. Working on the coordinate patch $X\backslash (H\cup E)\cong \mathbb C^n\backslash \{0\}$, we have
\be
\omega=i\partial\bar\partial u=\left(\frac{u'}{e^\rho}\delta_{jk}+(u''-u')\frac{\bar z^j z^k}{e^{2\rho}}\right)dz^j\wedge d\bar z^k,\nonumber
\ee
and
\be
\alpha=i\partial\bar\partial v=\left(\frac{v'}{e^\rho}\delta_{jk}+(v''-v')\frac{\bar z^j z^k}{e^{2\rho}}\right)dz^j\wedge d\bar z^k.\nonumber
\ee
With the above formulas, once can easily check that the eigenvalues of $\omega^{-1}\alpha$ are $\frac {v'}{u'}$ with multiplicity (n-1), and  $\frac{v''}{u''}$ with multiplicity one (for instance, see \cite{FL}).

 In fact, before we write down the deformed Hermitian-Yang-Mills equation in this setting, we can simplify our picture further. Because $u''>0$, the first derivative $u'$ is monotone increasing, allowing us to  view $u'$ as a real variable, denoted by $x$, which ranges from $1$ to $a$. We then write $v'$ as a graph $f$ over $x\in(1,a)$:
\be
f(x)=f(u'(\rho))=v'(\rho).\nonumber
\ee
Taking the derivative of both sides, we see by the chain rule
\be
f'(x)u''(\rho)=v''(\rho).\nonumber
\ee
Working in the coordinate $x$, the eigenvalues of $\omega^{-1}\al$ are 
\be
\frac{v'}{u'}=\frac{f}x\,(\text{with multiplicity}\,n-1)\qquad{\rm and}\qquad\frac{v''}{u''}=f'.\nonumber
\ee
Note that as $x\rightarrow 1$, then $\rho\rightarrow -\infty$,  while $x\rightarrow a$ implies $\rho\rightarrow \infty$. Thus the asymptotics of $v(\rho)$ given by \eqref{vinfinity} are equivalent to
\be
\lim_{x\rightarrow 1^+}f(x)=q,\qquad\lim_{x\rightarrow a^-}f(a)=p,\nonumber
\ee
and we extend $f(x)$   to the boundary $[1,a]$ by continuity.

We now reformulate our problem into this setup. Using the explicit formulas for the eigenvalues of $\omega^{-1}\alpha$, need to find a real function $f:[1,a]\rightarrow \mathbb R$ with boundary values $f(1)=q,$ and $f(a)=p$, satisfying the ODE
\be
\label{ODE1}
{\rm Im}\left(e^{-i\hat\theta}(1+i\frac fx)^{n-1}(1+i f')\right)=0.
\ee
Since $x$ is always positive, multiplying by $x^{n-1}$ will not change the equation, so we rewrite the ODE as
\be
{\rm Im}\left(e^{-i\hat\theta}(x+if)^{n-1}(1+i f')\right)=0.\nonumber
\ee
Observe that this ODE is exact
\bea
{\rm Im}\left(e^{-i\hat\theta}(x+if)^{n-1}(1+i f')\right)&=&{\rm Im}\left(e^{-i\hat\theta}\frac d{dx}\frac{(x+if)^{n}}n\right)\nonumber\\
&=&\frac d{dx}{\rm Im}\left(e^{-i\hat\theta}\frac{(x+if)^{n}}n\right)=0.\nonumber
\eea
Thus we are looking for a function $f(x)$ so that the graph $(x, f(x))$ lies on a level curve of 
\be
\label{DHYMexact}
\Phi(x,y):={\rm Im}\left(e^{-i\hat\theta}(x+iy)^n\right).
\ee
Figure \ref{fig:LevelSets} below shows a level set $\Phi(x,y)=c$ for some $c\neq 0$, in the case that $n=11$. The $n$ dotted lines represent the level set $\Phi(x,y)=0$. Thus we see $\Phi(x,y)=c$ consists of $n$ disjoint curves lying in alternating sectors, asymptotic to the lines given by $\Phi(x,y)=0$. Solutions to the deformed Hermitian-Yang-Mills equations are graphical portions of the level set that lie over $[1,a]$. Solutions of the equation for different branches can be found by rotating by $2\pi/n$.
\begin{figure}[h!]
  \includegraphics[width=3in]{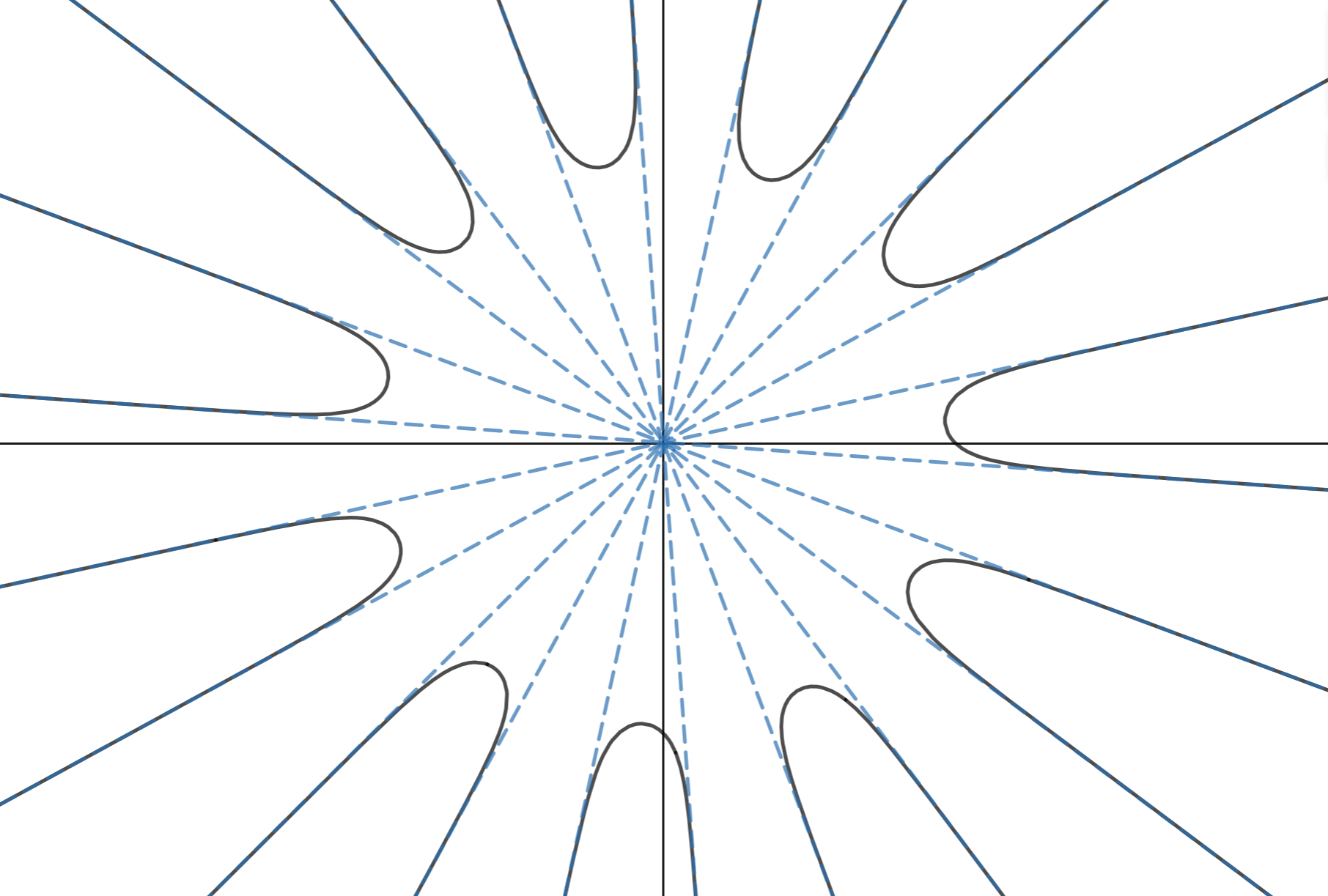}
  \caption{Graph of a level set $\Phi(x,y)=c$, in the case $n=11$.}
  \label{fig:LevelSets}
\end{figure}

\section{Stability}
\label{3}

We now turn to the stability condition that guarantees existence of a solution of \eqref{dHYM1}. This provides a coherent algebraic framework that is simple to interpret from initial conditions, without any assumptions on  explicit representatives of $[\omega]$ or $[\alpha]$. In this paper, we use ``central charge'' notation to highlight possible connections with Bridgeland stability conditions. We refer the reader to \cite{CY, CXY} for a more detailed discussion of stability and algebraic obstructions to solutions of the deformed Hermitian-Yang-Mills equations in general, and only focus in this paper on our specific geometric setup.

As stated in the introduction, for an analytic subvariety $V\subset X$, we define the following complex number:
\be
Z_{[\alpha][\omega]}(V):=-\int_V e^{-i\omega+\alpha},\nonumber
\ee
where by convention we only integrate the term in the expansion of order ${\rm dim}(V)$.

\begin{definition}
The pair $[\omega],[\alpha]$ is stable if, for each $k\in\{1,...,n-1\}$   all analytic subvarieties $V^k\subset X$ of dimension $k$ satisfy either
for all analytic subvarieties $V\subset X$,
\be
\label{stability}
{\rm Im}\left(\frac{Z_{[\alpha][\omega]}(V^k)}{Z_{[\alpha][\omega]}(X)}\right)>0 \qquad{\rm or}\qquad{\rm Im}\left(\frac{Z_{[\alpha][\omega]}(V^k)}{Z_{[\alpha][\omega]}(X)}\right)<0.
\ee
\end{definition}
This definition only makes sense if $Z_{[\alpha][\omega]}(X)\neq 0$, which is equivalent to our assumption that $\zeta_X\neq 0$. Now, because of our specific geometric setup, the inequality \eqref{stability} can be explicitly computed in terms of $a,p$, and $q$, for each analytic subvariety of $X$. 

Recall that $H$ is the pullback of the hyperplane divisor, and $E$ is the exceptional divisor, and that these divisors do no intersect.  We begin by computing $\zeta_X$ explicitly:
\bea
\zeta_X:=\int_X(\omega+i\alpha)^n&=&(a[H]-[E]+i(p[H]-q[E]))^n\nonumber\\
&=&(a+ip)^n[H]^n+(1+iq)^n(-1)^n[E]^n\nonumber\\
&=&(a+ip)^n-(1+iq)^n,\nonumber
\eea
where the last line follows since $[E]^n=(-1)^{n-1}$. Again by assumption $\zeta_X\neq 0$, which is the same as requiring    $a,p,$ and $q$ do not simultaneously satisfy 
\be
\label{zero}
|a+ip|=|1+iq|\qquad{\rm and}\qquad|{\rm arg}(a+ip)-{\rm arg}(1+iq)|=\frac {2\pi m}n
\ee
for some $m\in \mathbb Z$.  We remark that this does not provide a major constraint on which classes we consider. Given a choice of $q$, there are only a finite number of points $a+ip$ that satisfy $(a+ip)^n=(1+iq)^n$.

We now check stability for $H^{n-k}$ and $(-1)^{n-k-1}E^{n-k}$ for $k\in\{1,...,n-1\}$, where $k$ represents the dimension of each subvariety.  Here we multiply $E^{n-k}$ by $(-1)^{n-k-1}$ so that when this variety is viewed as a divisor of $(-1)^{n-k}E^{n-(k+1)}$ it is effective. We compute
\bea
Z_{[\alpha][\omega]}(H^{n-k})&=&-\int_{H^{n-k}}(-i)^k(\omega+i\alpha)^k\nonumber\\
&=&-\int_{H^{n-k}}i^{-k}(a[H]-[E]+i(p[H]-q[E]))^k\nonumber\\
&=&-i^{-k}(a+ip)^k[H]^k[H]^{n-k}\nonumber\\
&=&- i^{-k}(a+ip)^k.\nonumber
\eea
Next we see
\bea
Z_{[\alpha][\omega]}((-1)^{n-k-1}E^{n-k})&=&-\int_{(-1)^{n-k-1}E^{n-k}}(-i)^k(\omega+i\alpha)^k\nonumber\\
&=&-\int_{(-1)^{n-k-1}E^{n-k}}i^{-k}(a[H]-[E]+i(p[H]-q[E]))^k\nonumber\\
&=&-i^{-k}(-1)^k(1+iq)^k[E]^k(-1)^{n-k-1}[E]^{n-k}\nonumber\\
&=&- i^{-k}(-1)^{n-1}(1+iq)^k[E]^n\nonumber\\
&=&- i^{-k} (1+iq)^k,\nonumber
\eea
since as above $[E]^n=(-1)^{n-1}$. We also can compute the charge of our manifold $X$, and note
\be
Z_{[\alpha][\omega]}(X)=-\int_X(-i)^n(\omega+i\alpha)^n=-(i)^{-n}\zeta_X=-(i)^{-n}r_{X}e^{i\hat\theta},\nonumber
\ee
for some fixed real number $r_X$. Since $r_X>0,$ we can multiply \eqref{stability} by $r_X$ without changing the sign of the inequality, and so we note 
\be
r_X{\rm Im}\left(\frac{Z_{[\alpha][\omega]}(V^k)}{Z_{[\alpha][\omega]}(X)}\right)={\rm Im}\left(\frac{r_XZ_{[\alpha][\omega]}(V^k)}{-i^{-n}r_{X}e^{i\hat\theta}}\right)= {\rm Im}\left( -i^ne^{-i\hat\theta}{Z_{[\alpha][\omega]}(V^k)} \right).\nonumber
\ee
Thus, plugging in our formulas for $H^{n-k}$ and $(-1)^{n-k-1}E^{n-k}$ gives either
\be
{\rm Im}\left( i^{n-k}e^{-i\hat\theta}(a+ip)^k \right)>0,\nonumber
\ee
and
\be
{\rm Im}\left( i^{n-k}e^{-i\hat\theta}(1+iq)^k \right)>0,\nonumber
\ee
or the above with the inequality flipped. Summing up we have:
\begin{lemma}
\label{stabilitylem}
Given a choice of classes $[\omega]=a[H]-[E]$ and $[\alpha]=p[H]-q[E]$ on $X$, denote complex numbers $z_1=(1+iq)$ and $z_2=(a+ip)$. Then the pair $[\omega],[\alpha]$ is stable if and only if, for all $k\in\{1,...,n-1\}$, 
\be
\label{stability2}
{\rm Im}\left( i^{n-k}e^{-i\hat\theta}(z_\ell)^k \right)>0\qquad{\rm or}\qquad {\rm Im}\left( i^{n-k}e^{-i\hat\theta}(z_\ell)^k \right)<0
\ee
for  $\ell\in\{1,2\}$.
\end{lemma}

We now turn to some preliminary results about the structure of the inequalities defined in \eqref{stability2}. Let $z$ be the standard coordinate on $\mathbb C$, and choose a branch cut along the negative $x$-axis, so that $- \pi \leq {\rm arg} (z)< \pi$. For each $k\in\{1,...,n\}$, consider the  set defined by
\be
{\mathcal R}_k:=\{z\in\mathbb C|\,{\rm Im}\left( i^{n-k}e^{-i\hat\theta}z^k \right)=0\,\,{\rm and}\,\,- \frac{\pi}2 \leq {\rm arg} (z)<\frac{\pi}2\},\nonumber
\ee
which consists of $k$-rays emanating from the origin. Even though the stability conditions above are only defined for $k\leq n-1$, it is useful for our proof to also consider the rays determined by the $k=n$ case. Now, denote these rays via $\{r_k^1,r^2_k,..., r^k_k\},$ numbered so that
\be
\frac\pi2>{\rm arg}(r_k^1)>{\rm arg}(r_k^2)>\cdots> {\rm arg}(r_k^k)\geq -\frac{\pi}2.\nonumber
\ee
By definition of the map $z\mapsto z^k$, we see that these rays are all $\frac\pi k$ rotations of each other, i.e. ${\rm arg}(r_k^{j+1})-{\rm arg}(r_k^{j})=\frac\pi k$. Next, we define a {\it sector} to be the space between (but not including) two adjacent rays. Again, by the behavior of $z\mapsto z^k$, we see that the space
\be
{\mathcal S_k}:=\{z\in\mathbb C\,|\,{\rm Im}\left( i^{n-k}e^{-i\hat\theta}z^k \right)>0\,\,{\rm and}\,\,- \frac{\pi}2 \leq {\rm arg} (z)<\frac{\pi}2\}\nonumber
\ee
consists of alternating sectors, i.e. each ray bounds one and only one sector in ${\mathcal S_k}$. See Figure $\ref{fig:Sk}$ below.

 \begin{figure}[h!]
  \includegraphics[width=1.5in]{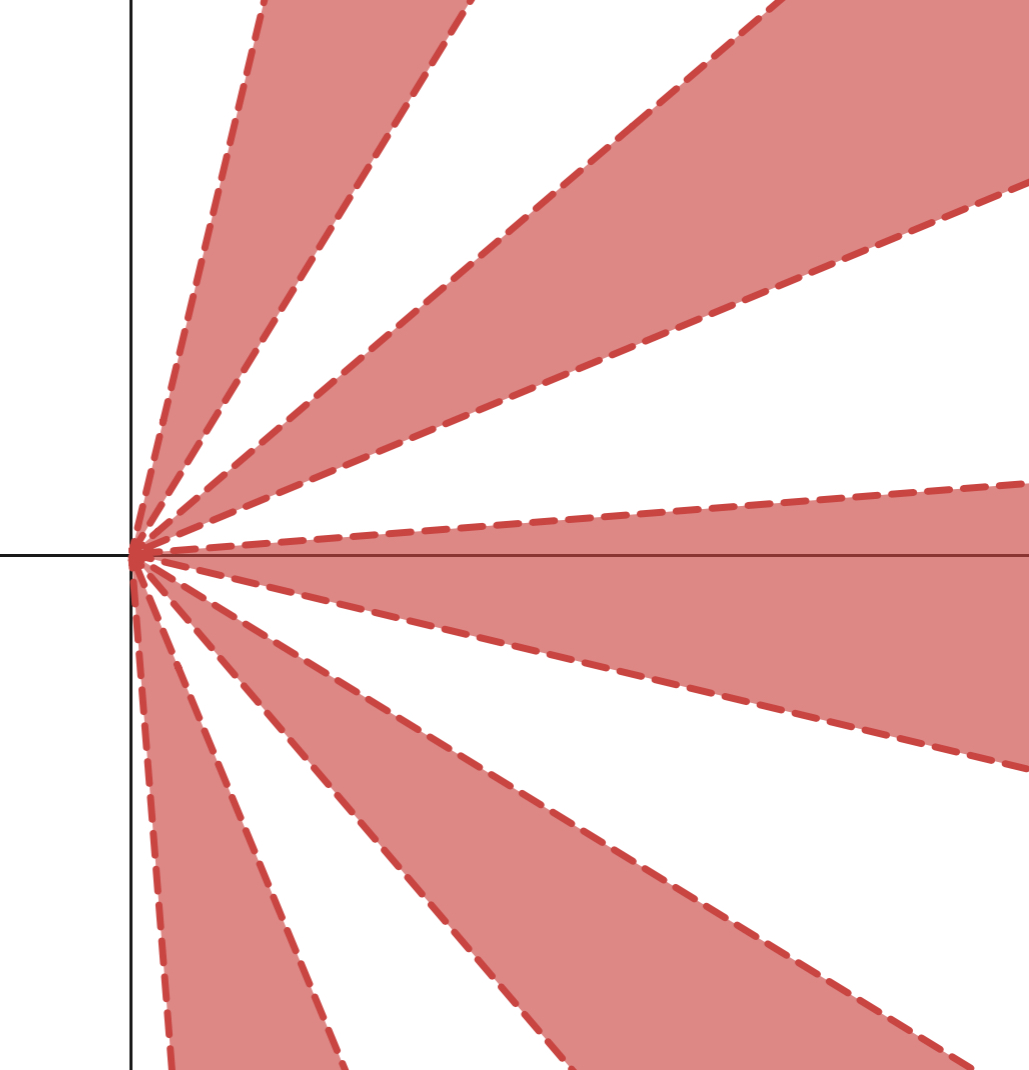}
  \caption{The set ${\mathcal S}_k$, in the case $k=10$.}
  \label{fig:Sk}
\end{figure}
Furthermore, consider the set
\be
{\mathcal S_k^-}:=\{z\in\mathbb C\,|\,{\rm Im}\left( i^{n-k}e^{-i\hat\theta}z^k \right)<0\,\,{\rm and}\,\,- \frac{\pi}2 \leq {\rm arg} (z)<\frac{\pi}2\}.\nonumber
\ee
Now, if we write a ray $r^j_k$ as $\mathbb R_+e^{i\phi_k^j}$, we see the sets of rays can be identified with sets of angles, i.e. ${\mathcal R}_k\cong\{\phi^1_k,...,\phi^k_k\}$. We conclude this section with a combinatorial argument that plays a key role in the proof of Theorem \ref{maintheorem}.

\begin{proposition}
\label{comb}
For any $k\in\{2,..., n\}$, the rays in the sets ${\mathcal R}_k$ and ${\mathcal R}_{k-1}$ alternate, and ${\mathcal R}_k$ contains the rays with the largest and smallest argument. In particular: 
\be
 \frac\pi2>\phi_k^1>\phi_{k-1}^1>\phi_k^2>\phi_{k-1}^2>\cdots>\phi_{k-1}^{k-2}>\phi_k^{k-1}>\phi_{k-1}^{k-1}\geq\phi_k^k\geq -\frac{\pi}2.\nonumber
\ee
Furthermore, if the last inequality is strict, i.e. $\phi_k^k >-\frac{\pi}2$, then $\phi_{k-1}^{k-1}>\phi_k^k$ as well.
\end{proposition}

 \begin{figure}[h!]
  \includegraphics[width=3in]{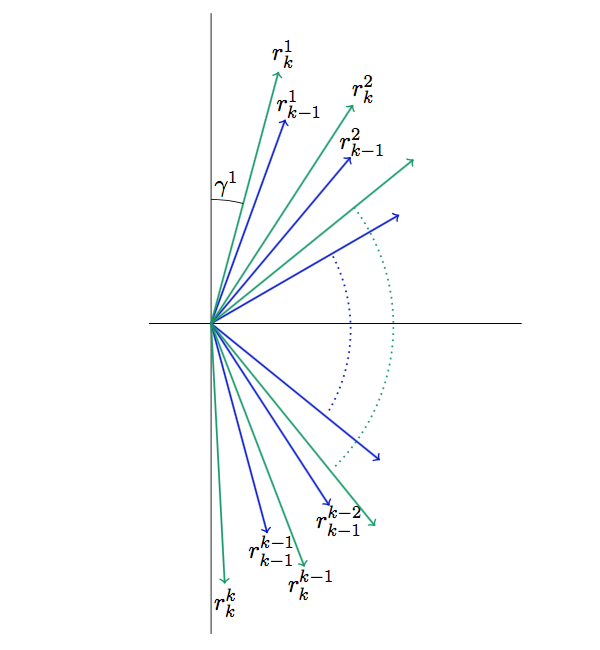}
  \caption{The alternating condition for rays in sets ${\mathcal R}_k$ and ${\mathcal R}_{k-1}$.}
  \label{fig:alt}
\end{figure}

\begin{proof}

Pick two angles $\phi^\ell_k$ and $\phi^j_{k-1}$ from ${\mathcal R}_k$ and ${\mathcal R}_{k-1}$, respectively. It will be convenient to express these angles by their distance to $\frac\pi2$, so we  set $\phi^\ell_k=\frac \pi2-\gamma^\ell$ and $\phi^j_{k-1}=\frac \pi2-\sigma^j$.

Now, since $\phi^\ell_k$ specifies a ray in the set   ${\mathcal R}_k$, by definition we  have
\be
{\rm Im}\left( e^{i\frac\pi2(n-k)}e^{-i\hat\theta}e^{ik\phi^\ell_k} \right)={\rm Im}\left( e^{i\frac\pi2(n-k)}e^{-i\hat\theta}e^{ik(\frac\pi2-\gamma^\ell)} \right)=0\nonumber
\ee
This equation holds if and only if
\be
\label{thing1}
\frac{n\pi}2-\hat\theta=k\gamma^\ell+q\pi
\ee
for some $q\in\mathbb Z$. Next, since $\phi^j_{k-1}$  lies in  ${\mathcal R}_k$ we have
\be
{\rm Im}\left( e^{i\frac\pi2(n-k+1)}e^{-i\hat\theta}e^{i(k-1)(\frac\pi2-\sigma^j)} \right)=0,\nonumber
\ee
which is equivalent to
\be
\frac{n\pi}2-\hat\theta-(k-1)\sigma^j=p\pi\nonumber
\ee
for some $p\in\mathbb Z$. Plugging in \eqref{thing1} gives that for all $\ell, j$, there exists an  $m\in\mathbb Z$ so that
\be
\label{thing2}
k\gamma^\ell -(k-1)\sigma^j=m\pi.
\ee
This is the key equation relating our angles $\phi^\ell_k$ and $\phi^j_{k-1}$.

First we prove the result in the special case that $\phi^k_k=-\frac\pi 2$. In this case   $\gamma^k=\pi$, and plugging this into \eqref{thing2} we see that $\sigma^{k-1}=\pi$ solves the equation for $m=1$. This implies $\phi_{k-1}^{k-1}=-\frac\pi2$ as well. To see the rays satisfy the alternation condition, note that all rays in ${\mathcal R}_k$ are $\frac\pi k$ rotations of each other, and furthermore both ${\mathcal R}_k$ and ${\mathcal R}_{k-1}$ contain the negative $y-$axis. As a result
\be
\phi^\ell_k=\frac\pi2-\frac{\ell\pi} k\qquad{\rm and}\qquad \phi^j_{k-1}=\frac\pi2-\frac{j\pi} {k-1},\nonumber
\ee
for $\ell\in\{1,...,k\}$ and $j\in\{1,...,k-1\}$, from which the alternating condition is clear.



We now turn to the general case, and assume that $\phi^k_k>-\frac\pi2$. As above write $\phi^1_k=\frac \pi2-\gamma^1$  and $\phi^1_{k-1}=\frac \pi2-\sigma^1$.  Since the rays in ${\mathcal R}_k$ are $\frac\pi k$ rotations of each other, and $\phi_k^1$ is the first ray to the right of the positive $y-$axis, we know $0<\gamma^1<\frac\pi k$ (since $\gamma^1=\frac\pi k$ corresponds to the special case  $\phi^k_k=-\frac\pi 2$). Similarly we know $0<\sigma^1<\frac\pi {k-1}$. Returning to \eqref{thing2}, and using that $k\gamma^1<\pi$, we know that for some $m\in \mathbb Z$
\be
\sigma^1=\frac{k\gamma^1-m\pi}{k-1}<\frac{\pi(1-m)}{k-1}.\nonumber
\ee
Since $\sigma^1>0$ we must have $m\leq 0$. Furthermore, using that $k\gamma^1>0$ gives
\be
\sigma^1=\frac{k\gamma^1-m\pi}{k-1}>\frac{-m\pi}{k-1}.\nonumber
\ee
Yet because we know $\sigma^1<\frac{\pi}{k-1}$, $m$ can not be strictly negative. Thus $m=0,$ giving
\be
\label{sigmaform}
\sigma^1=\frac{k\gamma^1}{k-1}.
\ee


Now that we have an equation specifying $\sigma^1$, we can write down the following general forms for our angles $\phi^\ell_k$ and $\phi^j_{k-1}$. Specifically, 
\be
\phi^\ell_k=\frac\pi2-\gamma^1-(\ell-1)\frac\pi k\qquad{\rm and}\qquad\phi^j_{k-1}=\frac\pi 2-\frac{k\gamma^1}{k-1}-(j-1)\frac\pi{k-1}.\nonumber
\ee
This is equivalent to
\be
\gamma^\ell=\gamma^1+(\ell-1)\frac\pi k\qquad{\rm and}\qquad\sigma^j = \frac{k\gamma^1}{k-1}+(j-1)\frac\pi{k-1}.\nonumber
\ee
For all $\ell, j$ this gives an explicit solution to \eqref{thing2}, with $m=\ell-j$. 

To complete the proof, we demonstrate the alternating condition, which states for $j\in\{1,...,k-1\}$,
\be
\phi^j_{k}>\phi^{j}_{k-1}>\phi^{j+1}_k.\nonumber
\ee
Using our explicit angle formulas this can be written as 
\be
-\gamma^1-(j-1)\frac\pi k> -\frac{k\gamma^1}{k-1}-(j-1)\frac\pi{k-1}> -\gamma^1-j\frac\pi k,\nonumber
\ee
which is equivalent to
\be
 (j-1)\frac\pi{k-1}-(j-1)\frac\pi k>\gamma^1 -\frac{k\gamma^1}{k-1} > (j-1)\frac\pi{k-1}-j\frac\pi k.\nonumber
\ee
Multiplying through by $k-1$ gives
\be
 (j-1)\pi-(j-1)\pi\frac{k-1} k>-\gamma^1  > (j-1)\pi-j\pi\frac{k-1} k.\nonumber
\ee
Simplifying, and multiplying by $-1$, we arrive at
\be
-\frac{(j-1)\pi}{k}<\gamma^1<\pi(\frac {k-j}k),\nonumber
\ee
which certainly holds for all $j\in\{1,...,k-1\}$, assuming that $0<\gamma^1<\frac\pi k$. This completes the proof of the proposition. 
\end{proof}

\section{Proof of  Theorem \ref{maintheorem}}
\label{4}

In this section we prove  our main result,  and construct a solution to the deformed Hermitian-Yang-Mills equation assuming stability of the pair $[\omega],[\alpha]$. 

Recall that on $X$ equation \eqref{dHYM1} on be reformulated using Calabi symmetry. Specifically we are looking for a real function $f:[1,a]\rightarrow \mathbb R$ with boundary values $f(1)=q,$ and $f(a)=p$, satisfying
\be
{\rm Im}\left(e^{-i\hat\theta}(1+i\frac fx)^{n-1}(1+i f')\right)=0.\nonumber
\ee
We saw above that this ODE is exact, and can be integrated to give level curves defined by \eqref{DHYMexact}. Thus we need a function $f$ that satisfies the boundary condition and lies on one of these level curves. For this to be possible, we need the specified boundary points $(1,q)$ and  $(a,p)$ to lie on the same level set.

\begin{lemma}
\label{levelset}
For any choice of  $[\omega]$ and $[\alpha]$, the fixed boundary points $(1,q)$ and  $(a,p)$  lie on the same level set of
\be
\Phi(x,y):={\rm Im}\left(e^{-i\hat\theta}(x+iy)^n\right)\nonumber
\ee
\end{lemma}
\begin{proof}
Recall  the complex number $\zeta_X=\int_X(\omega+i\alpha)^n$, which in our case is computed to be $(a+ip)^n-(1+iq)^n$. Set $\zeta_X=r_Xe^{i\hat\theta}$. Taking the complex conjugate gives $r_Xe^{-i\hat\theta}=(a-ip)^n-(1-iq)^n.$ Rearranging terms we see
\be
e^{-i\hat\theta}=\frac{(a-ip)^n-(1-iq)^n}{r_X}.\nonumber
\ee
We then have
\bea
\Phi(a,p)&=&{\rm Im}\left(\frac{(a-ip)^n-(1-iq)^n}{r_X}(a+ip)^n\right)\nonumber\\
&=&{\rm Im}\left(\frac{(a^2+p^2)^n}{r_X}-\frac{(a+ip)^n(1-iq)^n}{r_X}\right).\nonumber
\eea
The first term inside of the imaginary part above is real, so
\be
\Phi(a,p)=-{\rm Im}\left(\frac{(a+ip)^n(1-iq)^n}{r_X}\right).\nonumber
\ee
In exactly the same fashion we see
\be
\Phi(1,q)={\rm Im}\left(\frac{(a-ip)^n(1+iq)^n}{r_X}\right).\nonumber
\ee
Since ${\rm Im}(z)=-{\rm Im}(\bar z)$ it follows that $\Phi(a,p)=\Phi(1,q)$, which completes the proof of the lemma.
\end{proof}

Thus $(1,q)$ and  $(a,p)$ always lie on the same level set, which we denote by $\Phi(x,y)=\Phi(a,p)=\Phi(1,q)=c$. We now need to analyze when these points can be connected by a portion of the level set which stays graphical. Note that each level set is made up of  several components. If  $c=0$, then the level set consists of $n$ lines through the origin, each line   $\frac\pi n$ rotation of the next. Since $a>1>0$, in this case the points $a+ip$ and  $1+iq$ each lie on a ray in ${\mathcal R}_n$ (although we do not know yet if they lie on the same ray).

If $c\neq0$, then the level set looks like $n$ distinct curves lying in alternating sectors (see Figure \ref{fig:LevelSets}). In order for there to exists a function lying on a level curve connecting $(1,q)$ to $(a,p)$, the boundary points need to be on the same component of the level set, which we now prove.
\begin{proposition}
\label{mainprop}
If the classes $[\omega], [\alpha]$ are stable in the sense of Lemma \ref{stabilitylem}, then the points $(1,q)$ and $(a,p)$ both lie on the same component of the level set $\Phi(x,y)=c$.
\end{proposition}

\begin{proof} Set $z_1=(1+iq)$ and $z_2=(a+ip)$. We argue by contradiction, and assume that $z_1$ and $z_2$ do not lie on the same component of the level set.  As a first step we show that there exists a ray $r^j_{n-1}\in\mathcal R_{n-1}$ lying between $z_1$ and $z_2$. To see this, note that if $c=0$, then by assumption $z_1$ and $z_2$ lie on distinct rays in ${\mathcal R}_n$. Applying Proposition \ref{comb} for $k=n$ we see exists a ray $r^j_{n-1}\in\mathcal R_{n-1}$   between $z_1$ and $z_2$.

In the case that $c\neq 0$, the level set looks like $n$ distinct curves lying in alternating sectors with angle $\frac\pi n$. If $z_1$ and $z_2$ do not lie on the same component, since the components are in alternating sectors, there exists at least one empty sector   between the  sector containing $z_1$ and the sector containing  $z_2$. The boundary of this empty sector consists of two rays $r^{j+1}_n$ and $r^{j}_n$, and thus these two rays   lie between $z_1$ and $z_2$.  Applying Proposition \ref{comb} for $k=n$ proves existence of a ray $r^{j}_{n-1}$ between $r^{j+1}_n$ and $r^{j}_n$, and thus $r^{j}_{n-1}$ lies between $z_1$ and $z_2$. 

We now apply an induction argument and show that if there exists a ray $r^j_k\in\mathcal R_{k}$ lying between $z_1$ and $z_2$, then there exists a ray $r^{\ell}_{k-1}\in\mathcal R_{k-1}$ lying between $z_1$ and $z_2$ as well. Note that by the stability assumption, either $z_1$ and $z_2$ both lie in $\mathcal S_k$, or they both lie in   $\mathcal S_k^-$ (depending on whether the inequality is positive or negative). The key to this proposition is that in either case, the sets containing both $z_1$ and $z_2$ consists of alternating sectors. Specifically, given that there exists a ray $r^j_k$ lying between $z_1$ and $z_2$, then $z_1$ and $z_2$ must lie in different sectors of $\mathcal S_k$ (or  $\mathcal S_k^-$). Because these sectors alternate, there must be an empty sector between    $z_1$ and $z_2$. The boundary of this empty sector consists of   two rays in $\mathcal R_{k}$, which we denote by $r^{\ell+1}_k$ and $r^\ell_k$. These two rays lie between $z_1$ and $z_2$, and Proposition \ref{comb} gives that the ray $r^{\ell}_{k-1}$ lies between $z_1$ and $z_2$ as well.

Thus, given that there exists a  ray $r^{j}_{n-1}$  between $z_1$ and $z_2$, applying the induction argument $n-2$ times gives that the ray $r^1_1$ lies between $z_1$ and $z_2$. However, the ray $r^1_1$ divides the space $\{z\in\mathbb C|- \frac{\pi}2 \leq {\rm arg} (z)<\frac{\pi}2\}$ into two regions, ${\mathcal S_1}$ and ${\mathcal S_1}^c$. Thus it is impossible that $z_1$ and $z_2$ are both in   $\mathcal S_1 $ (or  $\mathcal S_1^-$), while also lying on opposite sides of $r^1_1$. This  gives a contradiction, proving the proposition.

We remark that the proof may end sooner in the special case that $r^1_1$ is the negative $y-$axis. In this case, the ray $r^2_2$ is also the negative $y-$axis (see the proof of Proposition \ref{comb}), so in fact the ray $r^1_2$ must divide the space $\{z\in\mathbb C|- \frac{\pi}2 \leq {\rm arg} (z)<\frac{\pi}2\}$ into two regions. Thus the contradiction occurs at this step, with $k=2$, rather than $k=1$.
\end{proof}

To finish the proof of the Theorem \ref{maintheorem}, we need to show that there exists a function $f(x)$ with $f(1)=q$ and $f(a)=p$, so that the graph of the function lies on the level curve $\Phi(x,y)=c$. We have just demonstrated that the points $(1,q)$ to $(a,p)$ lie on the same component of the level set $\Phi(x,y)=c$, so all that remains to be shown is that the level curve connecting $(1,q)$ to $(a,p)$ does not have vertical slope.

First, if $c=0$, then the level curves of $\Phi(x,y)=0$ consist of $n$ rays in ${\mathcal R}_n$. The above proposition shows that $(1,q)$ to $(a,p)$ lie on the same ray $r^j_n$. Since the ray never has vertical slope, in this case we see right away that there exists a linear function   $f(x)$ with $f(1)=q$ and $f(a)=p$, proving the theorem.

In general, the points where the tangent line to $\Phi(x,y)=c$ has vertical slope are given by
\be
\frac{\partial}{\partial y}\Phi(x,y)=\frac{\partial}{\partial y}{\rm Im}\left(e^{-i\hat\theta}(x+iy)^n\right)={\rm Im}\left(ine^{-i\hat\theta}(x+iy)^{n-1}\right)=0.\nonumber
\ee
Dividing by $n$ and writing $z=x+iy$, these points satisfy
\be
{\rm Im}\left(ie^{-i\hat\theta}z^{n-1}\right)=0,\nonumber
\ee
and so by definition of ${\mathcal R}_{n-1}$ we see they lie on a ray $r^j_{n-1}$ (see Figure \ref{fig:Intersection}). Thus in order to show that the level curve connecting $(1,q)$ to $(a,p)$ does not have vertical slope, the curve can not pass over a ray $r^j_{n-1}$. By our stability assumption, both $z_1$ and $z_2$  can not be on opposite sides of the ray $r^j_{n-1}$. As a result the level curve connecting $(1,q)$ to $(a,p)$ does not have vertical slope, and thus there exists a   $f(x)$ with $f(1)=q$ and $f(a)=p$ that solves the ODE \eqref{ODE1}. Thus we have demonstrated that if the classes $[\omega], [\alpha]$ are stable, a solution to the deformed Hermitian-Yang-Mills equation exists.  This concludes the proof of Theorem \ref{maintheorem}.
\begin{figure}[!h]
  \includegraphics[width=3in]{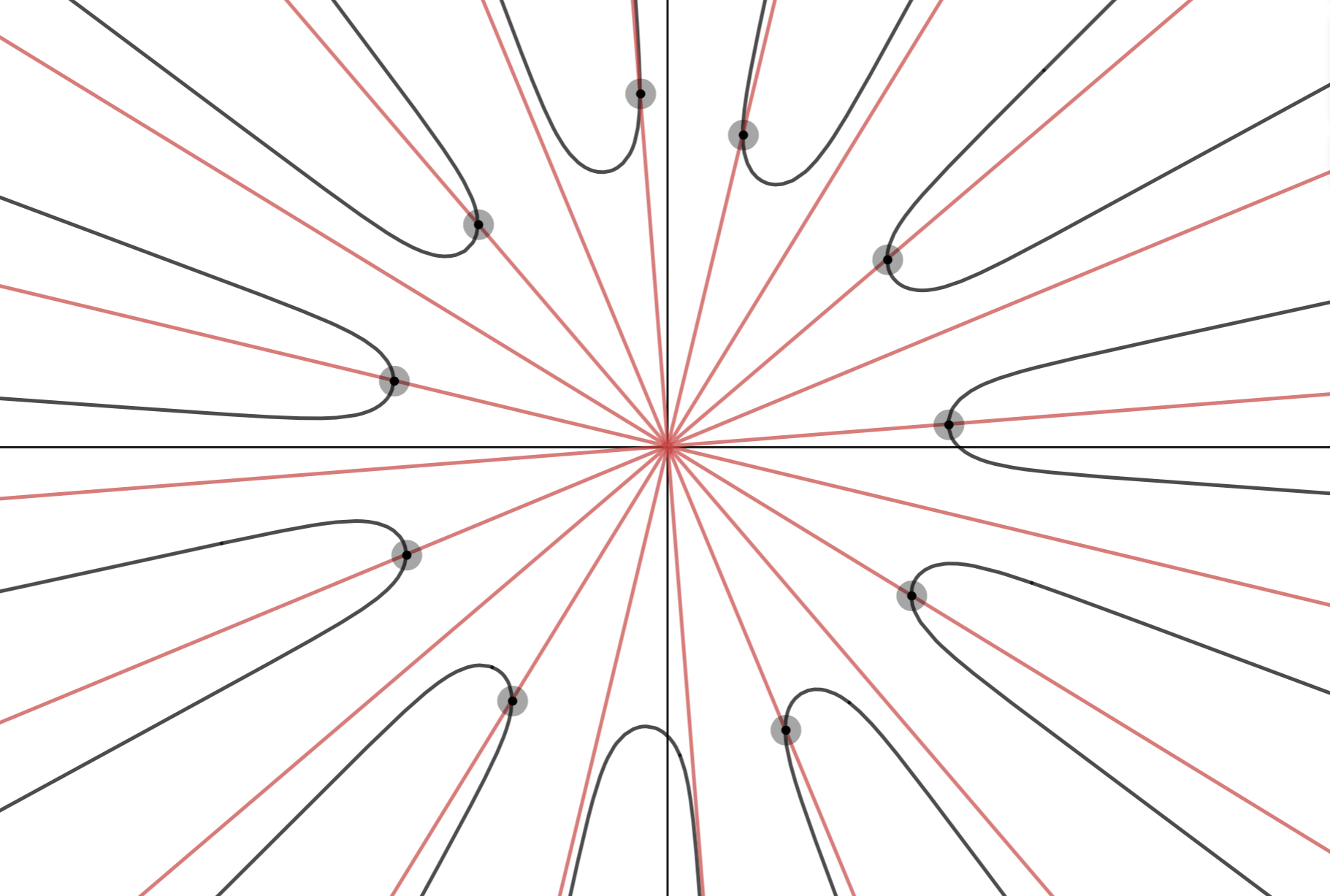}
  \caption{The intersection of a level set $\Phi(x,y)=c$ with the lines defined by ${\rm Im}\left(ie^{-i\hat\theta}z^{n-1}\right)=0$ occurs where the level set has vertical slope.}
  \label{fig:Intersection}
\end{figure}

\section{Lifting the average angle}
\label{5}

 Recall that the average angle  $\hat\theta$ is defined to be the argument of $\zeta_X=(a+ip)^n-(1+iq)^n$, which is a priori only $S^1$ valued (note that changing $\hat\theta$ by $2\pi$ does not effect equation \eqref{dHYM1}). This is in contrast to the pointwise angle $\Theta_\omega(\alpha)$, which as a sum of arctangents lifts to $\mathbb R$. Since \eqref{dHYM1} can be reformulated as  \eqref{dHYM2}, a solution to \eqref{dHYM2} specifies a unique lift of $\hat\theta$  to $\mathbb R$. A slightly weaker (but nevertheless analytic)  assumption to specify a lift would be the existence of a representative $\alpha_0$   for which  the point-wise angle $\Theta_\omega(\alpha_0)$ has oscillation  less that  $\pi$. This leads to the following question: is it possible to identify how $\hat\theta$ lifts to $\mathbb R$ from the initial data $a,p$ and $q$ alone, without needing to know existence of a specific representative of $[\alpha]$?

 
 In general the answer is not known, but there are special cases in which a lift exists. Collins-Xie-Yau consider the following situation in \cite{CXY}. Define a path $\gamma(t):[0,1]\rightarrow\mathbb C$ via
 \be
\gamma(t)=\int_X(\omega+it\alpha)^n.\nonumber
 \ee
At the starting time $\gamma(0)={\rm Vol}(X)=a^n-1$ is a positive real number,  which we  define to have zero argument. Also $\gamma(1)=\zeta_X$. Then, as long as $\gamma(t)\in \mathbb C^*$ for all $t\in[0,1]$, letting $t$ run from $0$ to $1$, we can count the number of times $\gamma(t)$ winds around the origin to define a lift of $\hat\theta$ to $\mathbb R$. 

Unfortunately there are examples where the angle $\hat\theta$ is well defined, but $\gamma(t)$ passes through the origin, so $\hat\theta$ can not be lifted using this method. We construct such an example in dimension $3$. First, fix a real number $q>\sqrt 3$.  Define an angle $\theta=\frac{2\pi}3-{\rm arctan}(q)$, and set $a=(\sqrt{q^2+1}){\rm cos}(\theta)$ and $p=-(\sqrt{q^2+1}){\rm sin}(\theta)$. Note that the choice $q>\sqrt 3$ ensures $a>1$. By construction $1+iq$ and $a+ip$ now satisfy \eqref{zero} for $k=1$, and therefore $(a+ip)^3=(1+iq)^3$. To complete our example, consider the initial data 
\be
[\omega]= a[H]-[E]\qquad{\rm and}\qquad [\alpha]= 2p[H]-2q[E],\nonumber
\ee
with $a$ and $p$ defined as above. Now, initially $\gamma(1)\neq 0$, since the arguments of $1+i2q$ and $a+i2p$ are  greater than $\frac{2\pi}3$ apart, while $\gamma(\frac12)=0$. Of course, one could always choose another path that avoids the origin, however then the lift will depend on the choice of the path. 

We remark that similar examples where the lift can not be defined exist in dimension 3 or higher. In dimension 2, the angle $\hat\theta$ always lifts, since the arguments of $1+itq$ and $a+itp$ can never be distance $\pi$ apart, so the path $\gamma(t)$ never passes through the origin. This is a special case of the fact that on a general K\"ahler surface, the angle $\hat\theta$ always lifts  by the Hodge Index Theorem \cite{CXY}.

One difficulty with the above method is that even if a lift of $\hat\theta$ exists,   in practice it can be hard to verify. Due to the specific geometry of our setup, we introduce a another notion of a lifted angle.

Assume that $\hat\theta$ lies in the branch cut $-\pi\leq \hat\theta<\pi$. 
Suppose that for a given choice of $[\omega]$ and $[\alpha]$, we have
\be
\label{nearby}
|{\rm arg}(a+ip)-{\rm arg}(1+iq)|<\frac{\pi}n.
\ee
We now lift $\hat\theta$ to $\mathbb R$ as follows. Construct two smooth strictly increasing functions $\rho_1(t),\rho_2(t):[0,1]\rightarrow[0,1]$, so that 
\be
|{\rm arg}(a+i\rho_1(t)p)-{\rm arg}(1+i\rho_2(t)q)|<\frac{\pi}n,\nonumber
\ee
and $ \rho_1(0)=\rho_2(0)=0$ while $\rho_1(1)=\rho_2(1)=1$. 
To see this can always be done, start with two points in a sector of angle $\frac{\pi}n$, then rotate that sector so it contains the positive $x-$axis. It is easy to see that during this rotation  the points can simultaneously be deformed to the positive $x-$axis in such a way that they stay within the sector, and their $x-$coordinate remains fixed. The reverse of this deformation determines the two functions $\rho_1(t),\rho_2(t)$. For all $t$ the complex numbers $(a+i\rho_1(t) p)^n$ and $(1+i\rho_2(t) q)^n$ lie in the same half-plane, and so the 
path $\ti\gamma(t)=(a+i\rho_1(t) p)^n-(1+i\rho_2(t) q)^n$ never passes through the origin and has  a winding number $k\in\mathbb Z$. We then define the lift of $\hat\theta$ (denoted $\hat\Theta_X$), by
\be
\label{lifted}
\hat\Theta_X:=\hat\theta+2\pi k \in(-n\frac \pi 2, n\frac\pi 2).
\ee
Again we emphasize that this lifted angle depends only on $a,p$ and $q$, and not on any representatives of the classes $[\omega]$ and $[\alpha]$. One advantage of using the above lifted angle is that our stability implies such a lift exists. 
\begin{proposition}
Suppose the pair $[\omega],[\alpha]$ is stable in the sense of Lemma \ref{stabilitylem}. Then the angle $\hat\theta$ has a well defined lift  $\hat\Theta_X$ given  by \eqref{lifted}.
\end{proposition}
\begin{proof}
By the induction argument given in Proposition \ref{mainprop}, we know from our stability assumption that the two points $(a+ip)$ and $(1+iq)$ can not have two rays from $\mathcal R_n$ between them. Since the rays in $\mathcal R_n$ are all $\frac\pi n$ rotations of each other, this verifies \eqref{nearby}, which allows us to define $\hat\Theta_X$.
\end{proof}

We expect that in general, being able to determine the lifted angle and specifying the branch will be a key step to solving the deformed Hermitian Yang Mills equation. This expectation is motivated by Theorem \ref{theorem2}, which shows the importance of the lifted angle in our specific case.

First, we note that for any subvariety $H^{n-k}$ or $(-1)^{n-k-1}E^{n-k}$, the lifted restricted angle is always well defined. Specifically, if we assume  $z_1=1+iq$ and $z_2=a+ip$ always have arguments in $(-\frac\pi2,\frac\pi2),$ then the lifted angle associated to each subvariety is given by
\be
\label{liftedsubvariety}
\hat\Theta_{(-1)^{n-k-1}E^{n-k}}=k\,{\rm arg}(z_1)\qquad{\rm and}\qquad \hat\Theta_{H^{n-k}}=k\,{\rm arg}(z_2).
\ee
We now present the proof of Theorem  \ref{theorem2}.

To begin, assume that for a given choice of $a, p, q$ there exists a lifted angle $\hat\Theta_X\in\mathbb R$. Furthermore assume that  $V^{n-1}$ (which can be either $H$ or $E$) satisfies 
\be
\label{strongstab}
\hat\Theta_X+\frac\pi2>\hat\Theta_{V^{n-1}} >\hat\Theta_X-\frac\pi2.
\ee
Using \eqref{liftedsubvariety} this implies 
\be
\frac{1}{n-1}\left(\hat\Theta_X+ \frac\pi2\right)>{\rm arg}(z_\ell) >\frac{1}{n-1}\left(\hat\Theta_X- \frac\pi2\right).\nonumber
\ee
for $\ell\in\{1,2\}$. Thus the difference between ${\rm arg}(z_1)$ and ${\rm arg}(z_2)$ is at most $\frac\pi{n-1}$. By Lemma \ref{levelset} both $z_1$ and $z_2$ lie on the same level set of $\Phi$, and since this level set consists of curves in alternating sectors with angle $\frac \pi n$, the angle bound of  $\frac\pi{n-1}$ tells us that either $z_1$ and $z_2$ lie in the same component of the level set, or they lie on two adjacent components with an empty sector in between. We must rule out the latter possibility.

Note that \eqref{strongstab} implies 
\be
\pi> \frac\pi2-\hat\Theta_X+(n-1){\rm arg}(z_\ell) >0.
\ee
This is equivalent to
\be
{\rm Im}\left(ie^{-i\hat\theta}(z_\ell)^{n-1}\right)>0\nonumber
\ee
for $\ell\in\{1,2\}$, which  is just stability in the sense of Lemma  \ref{stabilitylem} for  $k=n-1.$ So $z_1$ and $z_2$ lie in $\mathcal S_{n-1}$. Right away this rules out the possibility that they lie on distinct adjacent rays in $\mathcal R_n$, since any two such rays will never both be contained in $\mathcal S_{n-1}$. We can also rule out the case where $z_1$ and $z_2$ lie in two adjacent components which are not rays. In this case, there will be exactly two rays in $\mathcal R_n$ between $z_1$ and $z_2$, and thus by Proposition \ref{comb} at least one ray in $\mathcal R_{n-1}$. Yet because the sectors in $\mathcal S_{n-1}$ alternate, there must in fact be two rays in $\mathcal R_{n-1}$ between $z_1$ and $z_2$. But this is impossible if the difference between ${\rm arg}(z_1)$ and ${\rm arg}(z_2)$ is at most $\frac\pi{n-1}$.

Thus  $z_1$ and $z_2$ lie in the same component of the level set of $\Phi$. Furthermore, just as in the proof of Theorem  \ref{maintheorem}, stability in the sense of Lemma  \ref{stabilitylem} for  $k=n-1$  rules out the possibility of a vertical slope on the level curve connecting $z_1$ and $z_2$, and so a solution to the deformed Hermitian-Yang-Mills equation exists.

Conversely, suppose  for a given $a, p, q$ there exists a solution to equation  \eqref{dHYM2}. As explained above, because the pointwise angle is a sum of arctangents, a solution to  \eqref{dHYM2} specifies a uniques lift $\hat\Theta_X\in\mathbb R$. Additionally, restricting a solution to either  $H$ or $E$, we lose one arctangent from the sum that makes up the pointwise angle. Since the image of arctangent lies in $(-\frac\pi2,\frac\pi2)$, the average angle on each of these divisors must satisfy \eqref{strongstab}. For details see Lemma 8.2 in \cite{CJY}. This completes the proof of Theorem \ref{theorem2}.

We conclude the paper by noting the distinction between the stability from Conjecture \ref{theconjecture} and our    stability in the sense of Lemma  \ref{stabilitylem}. Although the original conjecture is only stated for the supercritical phase case, It is not too difficult to see, looking at the proof of Proposition 8.3 in \cite{CJY}, that  it can be generalized to any phase as 
\be
\hat\Theta_X+(n-k)\frac\pi2>\hat\Theta_{V^{k}} >\hat\Theta_X-(n-k)\frac\pi2,\nonumber
\ee
provided that all associated phase angles lift. Thus one difference we see right away is that Conjecture \ref{theconjecture} requires all lifted angles to exist, while this is not true of our stability.  Furthermore, we see the above inequality forces $z_1$ and $z_2$ between two rays, whereas  Lemma  \ref{stabilitylem} places them in alternating sectors. When $k=n-1$, the above inequality is a stronger condition than what arises from  Lemma  \ref{stabilitylem}. However, when $k<n-1$, the rays fail to match up. It would be interesting to explore this phenomenon more in the future.
 \end{normalsize}

\end{document}